\documentclass[12pt, reqno]{amsart}
\usepackage{amsmath, amsthm, amscd, amsfonts, amssymb, graphicx, color}
\usepackage[bookmarksnumbered, colorlinks, plainpages]{hyperref}
\hypersetup{colorlinks=true,linkcolor=red, anchorcolor=green, citecolor=cyan, urlcolor=red, filecolor=magenta, pdftoolbar=true}
\input{mathrsfs.sty}

\newtheorem{theorem}{Theorem}[section]
\newtheorem{lemma}[theorem]{Lemma}

\newtheorem{cor}[theorem]{Corollary}
\theoremstyle{definition}

\theoremstyle{remark}
\newtheorem{remark}[theorem]{\bf{Remark}}
\numberwithin{equation}{section}
\begin{document}

\title [Improved inequalities for numerical radius]{Improved inequalities for the numerical radius via Cartesian decomposition}

\author[P. Bhunia, S. Jana, M. S. Moslehian and K. Paul] {Pintu Bhunia$^1$, Suvendu Jana$^2$, Mohammad Sal Moslehian$^3$, \MakeLowercase{and} Kallol Paul$^1$}

\address{$^1$ Department of Mathematics, Jadavpur University, Kolkata 700032, West Bengal, India}
\email{pintubhunia5206@gmail.com}
\email{kalloldada@gmail.com}

\address{$^2$ Department of Mathematics, Mahisadal Girls College, Purba Medinipur 721628, West Bengal, India}
\email{janasuva8@gmail.com} 

\address{$^3$ Department of Pure Mathematics, Center of Excellence in Analysis on Algebraic Structures (CEAAS), Ferdowsi University of Mashhad, P. O. Box 1159, Mashhad 91775, Iran.}
\email{moslehian@um.ac.ir; moslehian@yahoo.com}

\renewcommand{\subjclassname}{\textup{2020} Mathematics Subject Classification}

\maketitle

\begin{abstract}
We develop various lower bounds for the numerical radius $w(A)$ of a bounded linear operator $A$ defined on a complex Hilbert space, which improve the existing inequality $w^2(A)\geq \frac{1}{4}\|A^*A+AA^*\|$. In particular, for $r\geq 1$, we show that
\begin{eqnarray*}\frac{1}{4}\|A^*A+AA^*\|
	\leq\frac{1}{2} \left( \frac{1}{2}\|\Re(A)+\Im(A)\|^{2r}+\frac{1}{2}\|\Re(A)-\Im(A)\|^{2r}\right)^{\frac{1}{r}}
	\leq w^{2}(A),\end{eqnarray*}
where $\Re(A)$ and $\Im(A)$ are the real and imaginary parts of $A$, respectively. Furthermore, we obtain upper bounds for $w^2(A)$ refining the well-known upper bound $w^2(A)\leq \frac{1}{2} \left(w(A^2)+\|A\|^2\right)$.  Separate complete characterizations for $w(A)=\frac{\|A\|}{2}$ and $w(A)=\frac{1}{2}\sqrt{\|A^*A+AA^*\|}$ are also given.
\end{abstract}

\section{Introduction}

The purpose of the present article is to obtain improvements of the existing well-known upper and lower bounds for the numerical radius of bounded linear operators acting on Hilbert spaces in terms of their real and imaginary parts. This is in a continuation of the study done in recent article \cite{P18}. Let us first introduce some notations and terminologies.
\smallskip

Let $\mathscr{H}$ be a complex Hilbert space with the inner product $\langle \cdot,\cdot \rangle $ and the corresponding norm $\|\cdot\|$ induced by the inner product. Let $ \mathbb{B}(\mathscr{H})$ denote the $C^*$-algebra of all bounded linear operators on $\mathscr{H}$ with the identity $I$. Let $A\in \mathbb{B}(\mathscr{H})$. We denote by $|A|=({A^*A})^{\frac{1}{2}}$ the positive square root of $A$, and  $\Re(A)=\frac{1}{2}(A+A^*)$ and $\Im(A)=\frac{1}{2\rm i}(A-A^*)$, respectively, stand for the real and imaginary parts of $A$. The numerical range of $A$, denoted as $W(A)$, is defined by $W(A)=\left \{\langle Ax,x \rangle: x\in \mathscr{H}, \|x\|=1 \right \}.$
We denote by $\|A\|$, $ c(A) $, and $w(A)$ the operator norm, the Crawford number, and the numerical radius of $A$, respectively. Recall that $$c(A)=\inf \left \{|\langle Ax,x \rangle|: x\in \mathscr{H}, \|x\|=1 \right \}$$ and $$w(A)=\sup \left \{|\langle Ax,x \rangle|: x\in \mathscr{H}, \|x\|=1 \right \}.$$ 
It is well known that the numerical radius $ w(\cdot)$ defines a norm on $\mathbb{B}(\mathscr{H})$ and is equivalent to the operator norm $\|\cdot\|$. In fact, the following double inequality holds:
\begin{eqnarray}\label{eqv}
\frac{1}{2} \|A\|\leq w({A})\leq\|A\|.
\end{eqnarray}
The inequalities in (\ref{eqv}) are sharp. The first inequality becomes equality if $A^2=0$, and the second one turns into equality if $A$ is normal. Over the years, many mathematicians have obtained various refinements of (\ref{eqv}), we refer the reader to \cite{AK, aab, FEK, MOS, OMI, SAH} and references therein. In particular, Kittaneh \cite{E} improved the inequalities in (\ref{eqv}) by establishing that 
\begin{eqnarray}\label{k5}
\frac{1}{4}\|A^*A+A{A}^*\|\leq w^2({A})\leq\frac{1}{2}\|A^*A+A{A}^*\|.
\label{d}\end{eqnarray}
\smallskip

In this paper, we obtain several refinements of the first inequality in (\ref{d}), in terms of $ \| \Re(A)+\Im(A)\|$ and $ \| \Re(A)-\Im(A)\|$. Furthermore, we obtain upper bounds for the numerical radius of bounded linear operators improving the existing inequality $w^2(A)\leq \frac{1}{2}\left (w(A^2)+\|A\|^2 \right)$ obtained by Dragomir \cite[Th. 1]{aaa}.

\section{Main Results}

\noindent We start our work with the following observation that for every $ A\in\mathbb{B}(\mathscr{H})$, 
 \begin{eqnarray}\label{eq1}
\frac{1}{4}\|A^*A+AA^*\|=\frac{1}{4} \left \| (\Re(A)+\Im(A))^2+(\Re(A)-\Im(A))^2 \right\|.
\end{eqnarray}
 
First by using the identity (\ref{eq1}), we obtain the following improvement of the first inequality in (\ref{d}).
 
\begin{theorem}\label{th1}
If $ A\in\mathbb{B}(\mathscr{H})$, then \begin{align*}
&\frac{1}{4}\|A^*A+AA^*\|\\
&\leq\frac{1}{4} \| \Re(A)+\Im(A)\|^2+ \frac{1}{4}\|\Re(A)-\Im(A)\|^2 \\
 &\leq\frac{1}{4}\| \Re(A)+\Im(A)\|^2+\frac{1}{4}\|\Re(A)-\Im(A)\|^2+\frac{1}{4}c^2(\Re(A)+\Im(A)) +\frac{1}{4}c^2(\Re(A)-\Im(A))\\ 
 &\leq w^2(A). 
 \end{align*} 
\end{theorem}
\begin{proof}
	
It follows from (\ref{eq1}) that	
 
 \begin{eqnarray*}
 	\frac{1}{4}\|A^*A+AA^*\|&=&\frac{1}{4}\| (\Re(A)+\Im(A))^2+(\Re(A)-\Im(A))^2\| \\ 
 	&\leq&\frac{1}{4} \| \Re(A)+\Im(A)\|^2+ \frac{1}{4} \|\Re(A)-\Im(A)\|^2.
 \end{eqnarray*}
This is the first inequality, and  the second follows trivially. 
 
Now we prove the third inequality.
 Let $x\in\mathscr{H}$ with $\|x\|=1$. Then from the Cartesian decomposition of $A$, we get
 \begin{eqnarray*}
 \mid\langle Ax,x\rangle\mid^2&=&\langle\Re(A)x,x\rangle^2+\langle\Im(A)x,x\rangle^2\\ &=&\frac{1}{2}\left ( \langle\Re(A)x,x\rangle+\langle\Im(A)x,x\rangle\right )^2+\frac{1}{2}\left( \langle\Re(A)x,x\rangle-\langle\Im(A)x,x\rangle\right)^2\\
 &=&\frac{1}{2}\langle(\Re(A)+\Im(A))x,x\rangle^2+\frac{1}{2}\langle(\Re(A)-\Im(A))x,x\rangle^2.
 \end{eqnarray*}
 Therefore, we have the following two inequalities:
 \begin{eqnarray}\label{eq2}
 \frac{1}{2}c^2(\Re(A)+\Im(A))+\frac{1}{2}\|\Re(A)-\Im(A)\|^2\leq w^2(A) \end{eqnarray} and
 \begin{eqnarray}\label{eq3}
 \frac{1}{2}c^2(\Re(A)-\Im(A))+\frac{1}{2}\|\Re(A)+\Im(A)\|^2\leq w^2(A). \end{eqnarray}
It follows from (\ref{eq2}) and (\ref{eq3}) that
 $$\frac{1}{4}\| \Re(A)+\Im(A)\|^2+\frac{1}{4}\|\Re(A)-\Im(A)\|^2+\frac{1}{4}c^2(\Re(A)+\Im(A)) +\frac{1}{4}c^2(\Re(A)-\Im(A))\leq w^2(A).$$
\end{proof}

 Clearly, Theorem \ref{th1} refines the first inequality in (\ref{d}).
Now, the following corollary is trivially inferred from Theorem \ref{th1}.

\begin{cor}\label{cor1}
If $ A\in\mathbb{B}(\mathscr{H})$, then
\begin{eqnarray*} 
\frac{1}{4}\|A^*A+AA^*\|+\frac{1}{4}c^2(\Re(A)+\Im(A)) +\frac{1}{4}c^2(\Re(A)-\Im(A)) \leq w^2(A).\end{eqnarray*}
\end{cor}

 
Also, the following result easily derived from (\ref{eq2}) and (\ref{eq3}).

\begin{cor}
If $ A\in\mathbb{B}(\mathscr{H})$, then $ w^2(A)\geq\max \left\lbrace\beta_1,\beta_2\right\rbrace$, where $$ \beta_1= \frac{1}{2}c^2(\Re(A)+\Im(A))+\frac{1}{2}\|\Re(A)-\Im(A)\|^2,$$ $$\beta_2= \frac{1}{2}c^2(\Re(A)-\Im(A))+\frac{1}{2}\|\Re(A)+\Im(A)\|^2.$$
 
\label{th2}\end{cor}

\begin{remark}\label{remi}
(i) We have 
{\small \begin{align*}
 	&\max \left\lbrace\beta_1,\beta_2\right\rbrace\\ 
 &=\frac{1}{2}\left\lbrace \frac{c^2(\Re(A)+\Im(A))+\|\Re(A)-\Im(A)\|^2+c^2(\Re(A)-\Im(A))+\|\Re(A)+\Im(A)\|^2}{2}\right\rbrace\\&\quad+\frac{1}{2}\left\lbrace\frac{\mid \|\Re(A)+\Im(A)\|^2-\|\Re(A)-\Im(A)\|^2+c^2(\Re(A)-\Im(A))-c^2(\Re(A)+\Im(A))\mid}{2}\right\rbrace\\&\geq\frac{1}{4}\left\lbrace c^2(\Re(A)+\Im(A))+c^2(\Re(A)-\Im(A))\right\rbrace+\frac{1}{4}\|(\Re(A)-\Im(A))^2+(\Re(A)+\Im(A))^2\|\\&\quad+\frac{1}{4}\left | \|\Re(A)+\Im(A)\|^2-\|\Re(A)-\Im(A)\|^2+c^2(\Re(A)-\Im(A))-c^2(\Re(A)+\Im(A)) \right | \\&=\frac{1}{4}\|A^*A+AA^*\|+\frac{1}{4}c^2(\Re(A)+\Im(A)) +\frac{1}{4}c^2(\Re(A)-\Im(A))\\&\quad+\frac{1}{4}\mid \|\Re(A)+\Im(A)\|^2-\|\Re(A)-\Im(A)\|^2+c^2(\Re(A)-\Im(A))-c^2(\Re(A)+\Im(A))\mid. \end{align*}}
Thus, 
{\small\begin{align*}
 w^2(A) \geq&\frac{1}{4}\|A^*A+AA^*\|+\frac{1}{4}c^2(\Re(A)+\Im(A)) +\frac{1}{4}c^2(\Re(A)-\Im(A))\\&+\frac{1}{4}\left | \|\Re(A)+\Im(A)\|^2-\|\Re(A)-\Im(A)\|^2+c^2(\Re(A)-\Im(A))-c^2(\Re(A)+\Im(A)) \right |. \end{align*} }
(ii) Also, we remark that Corollary \ref{th2} is stronger than the recently obtained inequality in \cite[Th. 2.3]{bib1}. 
\end{remark}


To prove the next refinement of the first inequality in (\ref{d}), we need the following lemma, which can be found in \cite[Th. 2.17]{P18}. 
\begin{lemma}\label{lem16}
	Let $ A,D\in\mathbb{B}(\mathscr{H})$. Then $$\|A+D\|^2\leq \|A\|^2+\|D\|^2+\frac{1}{2}\|A^*A+D^*D\|+w(A^*D)$$ and $$\|A+D\|^2\leq \|A\|^2+\|D\|^2+\frac{1}{2}\|AA^*+DD^*\|+w(AD^*).$$
\end{lemma}

\begin{theorem}\label{theor16}
	If $ A\in\mathbb{B}(\mathscr{H})$, then \begin{align*}
		&\frac{1}{4}\|A^*A+AA^*\|\\&\leq\frac{1}{4}\bigg\{ \frac{3}{2} \|\Re(A)+\Im(A)\|^4+\frac{3}{2}\|\Re(A)-\Im(A)\|^4
		+ \|\Re(A)+\Im(A)\|^2 \|\Re(A)-\Im(A)\|^2\bigg\}^{\frac{1}{2}}\\&\leq w^2(A).\end{align*}
\end{theorem}
\begin{proof}
It follows from (\ref{eq1}) that
	{\small\begin{align*}
		&\frac{1}{16}\|A^*A+AA^*\|^2\\&=\frac{1}{16}\| (\Re(A)+\Im(A))^2+(\Re(A)-\Im(A))^2\|^2\\&\leq\frac{1}{16}\left\lbrace\|\Re(A)+\Im(A)\|^4+\|\Re(A)-\Im(A)\|^4+\frac{1}{2}\|(\Re(A)+\Im(A))^4+(\Re(A)-\Im(A))^4\|\right\rbrace\\&\quad+\frac{1}{16} w((\Re(A)+\Im(A))^2(\Re(A)-\Im(A))^2), \qquad\qquad \qquad\qquad \qquad\qquad (\text{using Lemma \ref{lem16}})
		\\&\leq\frac{1}{16}\left\lbrace\|\Re(A)+\Im(A)\|^4+\|\Re(A)-\Im(A)\|^4+\frac{1}{2}\left(\|\Re(A)+\Im(A)\|^4+\|\Re(A)-\Im(A)\|^4\right)\right\rbrace\\&\quad+\frac{1}{16}\|\Re(A)+\Im(A)\|^2\|\Re(A)-\Im(A)\|^2 \\
		&=\frac{1}{16}\bigg\{ \frac{3}{2} \|\Re(A)+\Im(A)\|^4+\frac{3}{2}\|\Re(A)-\Im(A)\|^4
		+ \|\Re(A)+\Im(A)\|^2 \|\Re(A)-\Im(A)\|^2\bigg\}\\
		&\leq w^4(A),
	\end{align*}}
	where the last inequality is deduced from \eqref{eq2p} and \eqref{eq3p}.
\end{proof}

Now, we state a lemma.

\begin{lemma}\cite[Th. 2.2]{aab} Let $ A,D\in\mathbb{B}(\mathscr{H})$. Then $$\|A+D\|^2\leq 2 \max \left\lbrace\|A^*A+D^*D\|,\|AA^*+DD^*\|\right\rbrace.$$
\label{lem15}\end{lemma}

Based on the above lemma, we obtain the following refinement of the first inequality in (\ref{d}).

\begin{theorem}\label{thn16} If
 $ A\in\mathbb{B}(\mathscr{H})$, then $$\frac{1}{4}\|A^*A+AA^*\| \leq\frac{1}{2\sqrt{2}}\left(\| \Re(A)+\Im(A)\|^4+\| \Re(A)-\Im(A) \|^4\right)^{\frac{1}{2}} \leq w^2(A).$$
\end{theorem}

\begin{proof}
It follows from (\ref{eq1}) that
\begin{eqnarray*}
\frac{1}{4}\|A^*A+AA^*\|&=&\frac{1}{4}\| (\Re(A)+\Im(A))^2+(\Re(A)-\Im(A))^2\|\\&\leq&\frac{1}{2\sqrt{2}}\| (\Re(A)+\Im(A))^4+(\Re(A)-\Im(A))^4\|^{\frac{1}{2}}, \qquad(\text{by Lemma $\ref{lem15}$})\\&\leq&\frac{1}{2\sqrt{2}}\left(\| \Re(A)+\Im(A)\|^4+\| \Re(A)-\Im(A)\|^4\right)^{\frac{1}{2}}\\&\leq& w^2(A),
\end{eqnarray*}
where we deduce the last inequality from \eqref{eq2p} and \eqref{eq3p}.
\end{proof}

We observe here that the convexity of the function $f(t)={t}^2$ ensures that the first inequality in Theorem \ref{thn16} is better than the first inequality in Theorem \ref{th1}. Also, we observe that the second inequality in Theorem \ref{thn16} is better than the second inequality in Theorem \ref{theor16}.


In the next theorem, we obtain another improvement of (\ref{d}). First we note that (\ref{eq2}) and (\ref{eq3}) imply  the following two inequalities, respectively: 

\begin{eqnarray}\label{eq2p}
\frac{1}{2}\|\Re(A)-\Im(A)\|^2\leq w^2(A) \end{eqnarray} 
and
\begin{eqnarray}\label{eq3p}
\frac{1}{2}\|\Re(A)+\Im(A)\|^2\leq w^2(A). \end{eqnarray}
Now, by employing the convexity property of the function $f(t)=t^r, \, r\geq 1$, in the first inequality in Theorem \ref{th1} and using inequalities (\ref{eq2p}) and (\ref{eq3p}), we get the following inequality.

\begin{theorem}\label{thp}
	If $ A\in\mathbb{B}(\mathscr{H})$, then for $r\geq 1$,
	
	\begin{eqnarray*}\frac{1}{4}\|A^*A+AA^*\|
		&\leq&\frac{1}{2} \left( \frac{1}{2}\|\Re(A)+\Im(A)\|^{2r}+\frac{1}{2}\|\Re(A)-\Im(A)\|^{2r}\right)^{\frac{1}{r}}
		\leq w^{2}(A).\end{eqnarray*}
\end{theorem}

\begin{remark}
	Clearly, Theorem \ref{thp} is a generalization of Theorem \ref{thn16}.
	We would like to remark that the second inequality in Theorem \ref{thp} gives more refinement as $r$ increases. 
\end{remark}


To prove our next result, we need the following lemma, which can be found in \cite{U}.

\begin{lemma}\label{lem14}
Let $A,D \in\mathbb{B}(\mathscr{H})$ be positive. Then $$ \|A+D\|\leq\max\left\lbrace\|A\|,\|D\|\right\rbrace+\|AD\|^\frac{1}{2}.$$\end{lemma}

\begin{theorem}\label{theor17}
	If $ A\in\mathbb{B}(\mathscr{H})$, then
	\begin{align*}
& \frac{1}{4}\|A^*A+AA^*\|\\
&\leq\frac{1}{4}\left [ \max\left\lbrace\|\Re(A)+\Im(A)\|^2,\| \Re(A)-\Im(A)\|^2\right\rbrace + \|\Re(A)+\Im(A)\|\|\Re(A)-\Im(A)\|\right] \\
&\leq w^2(A).\end{align*}

\end{theorem}
\begin{proof}
Equality (\ref{eq1}) and Lemma \ref{lem14} ensure that
\begin{align*} &\frac{1}{4}\|A^*A+AA^*\|\\&=\frac{1}{4}\|(\Re(A)+\Im(A) )^2+(\Re(A)-\Im(A))^2\|\\&\leq\frac{1}{4}\left[\max\left\lbrace\|\Re(A)+\Im(A)\|^2,\| \Re(A)-\Im(A)\|^2\right\rbrace + \|(\Re(A)+\Im(A))^2(\Re(A)-\Im(A))^2\|^\frac{1}{2}\right]
	\\&\leq\frac{1}{4}\left[\max\left\lbrace\|\Re(A)+\Im(A)\|^2,\| \Re(A)-\Im(A)\|^2\right\rbrace + \|(\Re(A)+\Im(A))\|\|(\Re(A)-\Im(A))\|\right]\\&\leq w^2(A),
\end{align*}
in which we employ \eqref{eq2p} and \eqref{eq3p}.
\end{proof}

We now concentrate our study on the equality of the first inequality in (\ref{d}).

\begin{cor}\label{corp2}
	Let $A\in \mathbb{B}(\mathscr{H})$. If $w^2(A)=\frac{1}{4}\|A^*A+AA^*\|$, then the following assertions hold:\\\\
	(i) There exists a sequence $\{x_n\}$ in $\mathscr{H}$ with $\|x_n\|=1$ such that $$\lim_{n\rightarrow \infty }|\langle \Re(A)x_n,x_n\rangle|=\lim_{n\rightarrow \infty }|\langle \Im(A)x_n,x_n\rangle|.$$
	(ii) $\|\Re(A)+\Im(A)\|^2=\|\Re(A)-\Im(A)\|^2=\frac{1}{2}\|A^*A+AA^*\|$.
\end{cor}
\begin{proof}
	Let $w^2(A)=\frac{1}{4}\|A^*A+AA^*\|$. It follows from Theorem \ref{th1} that $c(\Re(A)+\Im(A))=c(\Re(A)-\Im(A))=0$. This implies that there exist sequences $\{y_n\}$ and $\{z_n\}$ in $\mathscr{H}$ with $\|y_n \|=\|z_n \|=1$	 such that $\lim_{n\rightarrow \infty }\langle (\Re(A)+\Im(A))y_n,y_n\rangle=0$ and $\lim_{n\rightarrow \infty }\langle (\Re(A)-\Im(A))z_n,z_n \rangle=0$. Thus (i) holds.
	
Also, from (i) of Remark \ref{remi}, we have $\|\Re(A)+\Im(A)\|^2=\|\Re(A)-\Im(A)\|^2$. In addition, we  conclude from Theorem \ref{thp}   that $\|\Re(A)+\Im(A)\|^2=\frac{1}{2}\|A^*A+AA^*\|$, which yields (ii).
\end{proof}

Considering the matrix $A=\left(\begin{matrix}
1&0\\
0& \rm i
\end{matrix}\right)$, we conclude that the converse of Corollary \ref{corp2}, is not true, in general.
\bigskip


\begin{remark}
Considering the following two examples, we observe that the bounds obtained in Theorems \ref{theor16} and \ref{theor17} (also, Theorems \ref{thn16} and \ref{theor17}) are not comparable, in general. \\
(i) Let $ A=\begin{pmatrix}
2+2i & 0 \\
0 & 0
\end{pmatrix}$. Then $ \Re(A)=\begin{pmatrix}
2 & 0 \\
0 & 0
\end{pmatrix}$ and $ \Im(A)=\begin{pmatrix}
2 & 0 \\
0 & 0
\end{pmatrix}$. Clearly, $ \|\Re(A)+\Im(A)\|=4$ and $ \|\Re(A)-\Im(A)\|=0$. By simple calculations, we have 
\begin{align*}
& \frac{1}{2\sqrt{2}}\left(\| \Re(A)+\Im(A)\|^4+\| \Re(A)-\Im(A) \|^4\right)^{\frac{1}{2}}=4\sqrt{2}\approx5.65685424949,\\
&\frac{1}{4}\bigg\{ \frac{3}{2} \|\Re(A)+\Im(A)\|^4+\frac{3}{2}\|\Re(A)-\Im(A)\|^4
+ \|\Re(A)+\Im(A)\|^2 \|\Re(A)-\Im(A)\|^2\bigg\}^{\frac{1}{2}}\\
&=2\sqrt{6}\approx4.89897948557,\\
&\frac{1}{4}\left [ \max\left\lbrace\|\Re(A)+\Im(A)\|^2,\| \Re(A)-\Im(A)\|^2\right\rbrace + \|\Re(A)+\Im(A)\|\|\Re(A)-\Im(A)\|\right]=4.
\end{align*}
(ii) Let $ A=\begin{pmatrix}
3+2i & 0 \\
0 & 4i
\end{pmatrix}$. Then $ \Re(A)=\begin{pmatrix}
3 & 0 \\
0 & 0
\end{pmatrix}$ and $ \Im(A)=\begin{pmatrix}
2 & 0 \\
0 & 4
\end{pmatrix}$. Therefore, $ \|\Re(A)+\Im(A)\|=5$ and $ \|\Re(A)-\Im(A)\|=4$. By simple calculations, we get 
\begin{align*}
& \frac{1}{2\sqrt{2}}\left(\| \Re(A)+\Im(A)\|^4+\| \Re(A)-\Im(A) \|^4\right)^{\frac{1}{2}}=\frac{1}{2\sqrt{2}}\sqrt{881}\approx10.4940459309,\\
& \frac{1}{4}\bigg\{ \frac{3}{2} \|\Re(A)+\Im(A)\|^4+\frac{3}{2}\|\Re(A)-\Im(A)\|^4
+ \|\Re(A)+\Im(A)\|^2 \|\Re(A)-\Im(A)\|^2\bigg\}^{\frac{1}{2}}\\
&=\frac{1}{4}\sqrt{\frac{17215}{10}}\approx10.37274071,\\
& \frac{1}{4}\left [ \max\left\lbrace\|\Re(A)+\Im(A)\|^2,\| \Re(A)-\Im(A)\|^2\right\rbrace + \|\Re(A)+\Im(A)\|\|\Re(A)-\Im(A)\|\right]\\
&=\frac{45}{4}=11.25.
\end{align*}

\end{remark}


Now, we obtain an upper bound for the numerical radius of bounded linear operators. The following inequality is known as the Buzano inequality.

\begin{lemma}$($\cite{S}$)$ 
	Let $ x,y,e\in \mathscr{H}$ with $ \|e\|=1$. Then
	$$ \mid\langle x,e\rangle\langle e,y\rangle\mid\leq\frac{1}{2}\left(\mid\langle x,y\rangle \mid +\|x\|\|y\|\right).$$\label{lem11}\end{lemma}
\begin{theorem}
	Let $ A\in\mathbb{B}(\mathscr{H})$. Then $$ w^2(A)\leq\frac{1}{2}\left[\|A\|^2 \left (\min_{t\in[0,1]}\|t A^*A+(1-t)AA^*\| \right )+w^2(A^2)+w(A^2) \|A^*A+AA^*\|\right]^{\frac{1}{2}}.$$
	\label{th13}\end{theorem}

\begin{proof}
	Let $x\in\mathscr{H}$ with $\|x\|=1$. Then  
	\begin{align*}&\mid\langle Ax,x \rangle\mid^2\\&=\mid\langle Ax,x \rangle\langle x,A^*x \rangle\mid\\&\leq\frac{1}{2}\left( \mid \langle Ax,A^*x\rangle \mid +\|Ax\|\|A^*x\|\right) \,\,\qquad\qquad\qquad\qquad(\text{using Lemma \ref{lem11}})\\&=\frac{1}{2}\left\lbrace | \langle Ax,A^*x\rangle | ^2+\|Ax\|^2\|A^*x\|^2+2 | \langle Ax,A^*x\rangle | \|Ax\|\|A^*x\|\right\rbrace^{\frac{1}{2}}\\&\leq\frac{1}{2}\left\lbrace | \langle A^2x,x\rangle | ^2+\langle A^*Ax,x \rangle\langle AA^*x,x\rangle +|\langle A^2x,x\rangle | \langle (AA^*+ A^*A)x,x \rangle\right\rbrace^{\frac{1}{2}}\\&=\frac{1}{2}\Bigg \lbrace |\langle A^2x,x\rangle|^2+\langle A^*Ax,x \rangle^t\langle AA^*x,x\rangle^{1-t} \langle A^*Ax,x \rangle^{1-t}\langle AA^*x,x\rangle^t\\
		&\quad + |\langle A^2x,x\rangle | \langle (AA^*+ A^*A)x,x \rangle\Bigg \rbrace^{\frac{1}{2}}\\&\leq\frac{1}{2}\Bigg \lbrace |\langle A^2x,x\rangle|^2+\langle (tA^*A+(1-t) AA^*)x,x\rangle \langle ((1-t)A^*A+tAA^*)x,x\rangle\\
		&\quad + |\langle A^2x,x\rangle|\langle (AA^*+ A^*A)x,x \rangle\Bigg \rbrace^{\frac{1}{2}}\,\qquad \qquad\qquad\quad\text{(by the McCarthy inequality)}\\
		&\leq\frac{1}{2}\Bigg \lbrace |\langle A^2x,x\rangle|^2+\| tA^*A+(1-t) AA^*\| \| (1-t)A^*A+tAA^*\|\\
		& +|\langle A^2x,x\rangle | \langle (AA^*+ A^*A)x,x \rangle\Bigg \rbrace^{\frac{1}{2}}\\&\leq\frac{1}{2}\Bigg\lbrace |\langle A^2x,x\rangle|^2+\| tA^*A+(1-t) AA^*\| \|A\|^2+|\langle A^2x,x\rangle|\langle (AA^*+ A^*A)x,x \rangle\Bigg \rbrace^{\frac{1}{2}}\\&\leq\frac{1}{2}\Bigg \lbrace w^2(A^2)+\| tA^*A+(1-t) AA^*\| \|A\|^2+w( A^2)\| AA^*+ A^*A \|\Bigg \rbrace^{\frac{1}{2}}.
	\end{align*}
	Taking supremum over $\|x\|=1$, we get
	$$ w^2(A)\leq\frac{1}{2}\left[\|A\|^2 \left (\|t A^*A+(1-t)AA^*\| \right )+w^2(A^2)+w(A^2) \|A^*A+AA^*\|\right]^{\frac{1}{2}}.$$
	This holds for all $t\in [0,1]$, so considering minimum over $t\in [0,1]$, we have
	$$ w^2(A)\leq\frac{1}{2}\left[\|A\|^2 \left (\min_{t\in[0,1]}\|t A^*A+(1-t)AA^*\| \right )+w^2(A^2)+w(A^2) \|A^*A+AA^*\|\right]^{\frac{1}{2}},$$ as required.
\end{proof}

\begin{remark}
(i)	Dragomir \cite[Th. 1]{aaa} proved that for $A\in \mathbb{B}(\mathscr{H})$,
\begin{eqnarray}\label{d1}
w^2(A) \leq \frac{1}{2}\Big(\|A\|^2+w(A^2)\Big).
\end{eqnarray} 
		
	We would like to remark that the inequality in Theorem \ref{th13} is sharper than that in Dragomir's result \cite[Th. 1]{aaa}.\\
(ii)	Now, we consider $A=\left(\begin{matrix}
	0&1&0\\
	0&0&1\\
	0&0&0
	\end{matrix}\right)$. Then, by simple calculations, we arrive at 
	$$\frac{1}{2}\left[\|A\|^2 \left (\min_{t\in[0,1]}\|t A^*A+(1-t)AA^*\| \right )+w^2(A^2)+w(A^2) \|A^*A+AA^*\|\right]^{\frac{1}{2}}=\frac{3}{4}$$
	and $$\frac{1}{2}\|A^*A+AA^*\|=1.$$
Again,	considering another example $A=\left(\begin{matrix}
	0&2&0\\
	0&0&0\\
	0&0&\sqrt{2}
	\end{matrix}\right)$, we have 
	$$\frac{1}{2}\left[\|A\|^2 \left (\min_{t\in[0,1]}\|t A^*A+(1-t)AA^*\| \right )+w^2(A^2)+w(A^2) \|A^*A+AA^*\|\right]^{\frac{1}{2}}=\sqrt{5}$$
	and $$\frac{1}{2}\|A^*A+AA^*\|=2.$$
	Thus, we conclude that the second inequality in (\ref{d}) and our obtained inequality in Theorem \ref{th13} are not comparable, in general.\\
	
\end{remark}


In the following theorem, we present another refinement of (\ref{d1}). 

\begin{theorem}
	Let $ A\in\mathbb{B}(\mathscr{H})$. Then $$ w^4(A)\leq\frac{1}{4}\left[ w^2(A^2)+\frac{1}{4} \left \|(A^*A)^2+(AA^*)^2 \right \| +\frac{1}{2}w(A^*A^2A^*) +w(A^2) \|A^*A+AA^*\|\right].$$
	\label{th14}\end{theorem}

\begin{proof}
Let $x\in \mathscr{H}$ with $\|x\|=1$. It follows from Lemma \ref{lem11} that
\begin{eqnarray*}
	\langle A^*Ax,x \rangle\langle AA^*x,x\rangle
&=&\langle A^*Ax,x \rangle\langle x,AA^*x\rangle\\
	&\leq& \frac{\|A^*Ax\| \|AA^*x\|+ | \langle AA^*x,A^*Ax\rangle| }{2}\\
	&\leq& \frac{1}{4}\left(\|A^*Ax\|^2+ \|AA^*x\|^2 \right)+\frac{1}{2} | \langle A^*A^2A^*x,x\rangle|\\
	&=& \frac{1}{4}\left \langle \left ( (A^*A)^2 +(AA^*)^2 \right)x,x \right \rangle +\frac{1}{2} | \langle A^*A^2A^*x,x\rangle|\\
	&\leq& \frac{1}{4} \left \| (A^*A)^2 +(AA^*)^2 \right \| +\frac{1}{2}w( A^*A^2A^*).
\end{eqnarray*}
Following the proof of Theorem \ref{th13}, we infer that
	\begin{align*}
		 \mid &\langle Ax,x \rangle\mid^4\\
		&\leq\frac{1}{4}\left\lbrace | \langle A^2x,x\rangle | ^2+\langle A^*Ax,x \rangle\langle AA^*x,x\rangle +|\langle A^2x,x\rangle | \langle (AA^*+ A^*A)x,x \rangle\right\rbrace\\
		&\leq \frac{1}{4} \left \lbrace w^2(A^2)+\langle A^*Ax,x \rangle\langle AA^*x,x\rangle+ w(A^2)\|AA^*+ A^*A\| \right \rbrace\\
			&\leq \frac{1}{4} \left \lbrace w^2(A^2)+ \frac{1}{4} \left \| (A^*A)^2 +(AA^*)^2 \right \| +\frac{1}{2}w( A^*A^2A^*)+ w(A^2)\|AA^*+ A^*A\| \right \rbrace.
	\end{align*}
	Considering supremum over $\|x\|=1$, we arrive at the desired inequality.
\end{proof}

\begin{remark}
Clearly, for $A\in \mathbb{B}(\mathscr{H})$,	we have
\begin{align*}
\frac{1}{4}&\left[ w^2(A^2)+\frac{1}{4} \left \|(A^*A)^2+(AA^*)^2 \right \| +\frac{1}{2}w(A^*A^2A^*) +w(A^2) \|A^*A+AA^*\|\right]\\
&\leq \frac{1}{4}\left[ w^2(A^2)+\frac{1}{2} \left \|A\right \|^4 +\frac{1}{2} \left \| A^*A^2A^* \right \| +2w(A^2) \|A\|^2\right]\\
&\leq \frac{1}{4}\left[ w^2(A^2)+\frac{1}{2} \left \|A\right \|^4 +\frac{1}{2} \left \| A \right \|^4 +2 w(A^2) \|A\|^2\right]\\
&= \left[ \frac{ \|A \|^2+ w(A^2) }{2} \right]^2. 
\end{align*}
Therefore, Theorem \ref{th14} refines inequality (\ref{d1}).
\end{remark}

In our next theorem, we obtain an inequality involving norm and numerical radius of a bounded linear operator. First we recall the following well-known identity from \cite[p. 85]{YAM1}:
\begin{eqnarray}\label{ya}
w(A)=\sup_{\theta\in \mathbb{R} } \left \|\Re(e^{\rm i \theta}A) \right \|.
\end{eqnarray}

\begin{theorem}\label{thp1}
Let $A\in \mathbb{B}(\mathscr{H})$. Then
\[w^3(A)\leq \frac{1}{4} \Big[ w(A^3)+\|A\| \|A^2\|+w(A) \|A^*A+AA^*\| \Big].\]
\end{theorem}

\begin{proof}
By a short calculation, we get
\[\Re^3(A)=\frac{1}{4} \Re(A^3)+\frac{1}{8}(A^2A^*+{A^*}^2A)+\frac{1}{4}(A^*A+AA^*)\Re(A).\]
\end{proof}
Since $\Re(A)$ is selfadjoint, we have
\begin{eqnarray*}
\|\Re(A)\|^3&=& \left \|\frac{1}{4} \Re(A^3)+\frac{1}{8}(A^2A^*+{A^*}^2A)+\frac{1}{4}(A^*A+AA^*)\Re(A) \right \|\\
&\leq& \frac{1}{4} \left \| \Re(A^3) \right \| +\frac{1}{8}\left \| A^2A^*+{A^*}^2A \right \|+\frac{1}{4}\| A^*A+AA^*\| \|\Re(A)\|\\
&\leq& \frac{1}{4} \left \| \Re(A^3) \right \| +\frac{1}{4}\left \| A^2 \right \| \|A\|+\frac{1}{4}\| A^*A+AA^*\| \|\Re(A)\|.
\end{eqnarray*}
Now let $\theta\in \mathbb{R}$. Replacing $A$ with $e^{\rm i \theta}A$ in the last inequality yields that
\begin{eqnarray*}
	\|\Re(e^{\rm i \theta}A)\|^3 &\leq& \frac{1}{4} \left \| \Re(e^{3\rm i \theta}A^3) \right \| +\frac{1}{4}\left \| A^2 \right \| \|A\|+\frac{1}{4}\| A^*A+AA^*\| \|\Re(e^{\rm i \theta}A)\|.
\end{eqnarray*}
Taking supremum over all $\theta\in \mathbb{R}$ and using identity (\ref{ya}), we derive that 
\[w^3(A)\leq \frac{1}{4} \Big[ w(A^3)+\|A\| \|A^2\|+w(A) \|A^*A+AA^*\| \Big],\]
as desired.

\begin{remark}
Let $A\in \mathbb{B}(\mathscr{H})$ with $A\neq 0$ and $A^2=0$. It follows from Theorem \ref{thp1} that 
\[w(A)\leq\frac{1}{2} \sqrt{\|A^*A+AA^*\|}. \]
\end{remark}
\noindent This inequality combined with the first inequality in (\ref{k5}) ensures that
\[w(A)=\frac{1}{2}\sqrt{ \|A^*A+AA^*\|}. \] 
It should be mentioned that the reverse part is not true, in general. To see this, consider the matrix $A=\left(\begin{matrix}
0&1&0\\
0&0&1\\
0&0&0
\end{matrix}\right)$. Then one can easily verify that $w(A)=\frac{1}{\sqrt{2}}=\frac{1}{2} \sqrt{\|A^*A+AA^*\|}$, but $A^2\neq 0$.

At the end of the article, we give separate complete characterizations for $w(A)=\frac{1}{2}\|A\|$ and $w(A)=\frac{1}{2}\sqrt{\|A^*A+AA^*\|}$.
First we need the following lemma. Its proof  can be found in \cite[Th. 2.14]{BPaul}.

\begin{lemma}\label{prop12}
	Let $A\in \mathbb{B}(\mathscr{H})$. Then \\
	(i) $w(A)=\frac{1}{2}\|A\|$ if and only if $\overline{W(A)}$ is a circular disk with center at the origin and radius $\frac{1}{2}\|A\|$;\\
	(ii) $w(A)=\frac{1}{2}\sqrt{\|A^*A+AA^*\|}$ if and only if $\overline{W(A)}$ is a circular disk with center at the origin and radius $\frac{1}{2}\sqrt{\|A^*A+AA^*\|}$.
\end{lemma}

\begin{theorem}
	Let $A\in \mathbb{B}(\mathscr{H})$. Then 
	\begin{itemize}
	\item[(i)] $w(A)=\frac{1}{2}\|A\|$ if and only if $w(A+\lambda I)=\frac{1}{2}\|A\|+|\lambda|$ for all $\lambda \in \mathbb{C}$;
	\item[(ii)] $w(A)=\frac{1}{2}\sqrt{\|A^*A+AA^*\|}$ if and only if $w(A+\lambda I)=\frac{1}{2}\sqrt{\|A^*A+AA^*\|}+|\lambda|$ for all $\lambda \in \mathbb{C}$.
	\end{itemize}
\end{theorem}

\begin{proof}
	(i) The sufficient part is trivial,  so we only prove the necessary part. Let $w(A)=\frac{1}{2}\|A\|$. Clearly, $\overline{W(A+\lambda I)}= \overline{W(A)}+\lambda$ for all $\lambda \in \mathbb{C}$. Therefore, it follows from Lemma \ref{prop12}(i) that $\overline{W(A+\lambda I)}$ is a circular disk with center at $\lambda$ and radius $\frac{1}{2}\|A\|$. This implies that $w(A+\lambda I)=\frac{1}{2}\|A\|+|\lambda|$.\\
	(ii) The proof follows as in (i).
\end{proof}

\begin{remark} Let $A\in \mathbb{B}(\mathscr{H})$.
	 We would like to remark that if $\overline{W(A)}$ is a circular disk with center at the origin, then $w(A+\lambda I)=w(A)+|\lambda|$ for all $\lambda \in \mathbb{C}$. Hence, it follows from \cite[Lemma 2.13]{BPaul} that if $\left \| \Re (e^{\rm i \theta}A) \right \|=k\, (\textit{a constant})$ for all $\theta \in \mathbb{R}$, then $w(A+\lambda I)=w(A)+|\lambda|$ for all $\lambda \in \mathbb{C}$. This shows that if $\left \| \Re (e^{\rm i \theta}A) \right \|=k\, (\textit{a constant})$ for all $\theta \in \mathbb{R}$, then $w(A+\lambda I)\geq w(A)$ for all $\lambda \in \mathbb{C}$, that is, $A$ is Birkhoff--James numerical radius orthogonal to $I$, (for the details of Birkhoff--James numerical radius orthogonality, we refer to \cite{MPS, ZAM}). Finally, we remark that if either $w(A)=\frac{1}{2}\|A\|$ or $w(A)=\frac{1}{2}\sqrt{\|A^*A+AA^*\|}$, then $A$ is Birkhoff--James numerical radius orthogonal to $I$.
\end{remark}

\bibliographystyle{amsplain}

\end{document}